\documentclass[12pt]{article}

\usepackage{amsfonts}
\usepackage{amssymb}
\usepackage{amsthm}
\usepackage{amsmath}
\usepackage{amscd}
\usepackage{mathrsfs}
\usepackage{enumerate}
\usepackage{ulem}
\usepackage{fontenc}
\usepackage[all]{xy}
\usepackage{array}
\usepackage[english]{babel}

\usepackage{comment}

\newcommand{\SL} {\mathrm{SL}}
\newcommand{\SO} {\mathrm{SO}}
\newcommand{\Sp} {\mathrm{Sp}}

\renewcommand{\sl} {\mathrm{sl}}
\newcommand{\so}   {\mathrm{so}}
\newcommand{\oo}  {\mathrm{o}}
\renewcommand{\sp}   {\mathrm{sp}}

\newcommand{\Int} {\mathrm{Int}}

\newcommand{\im} {\mathrm{im}}
\newcommand{\Hom} {\mathrm{Hom}}
\newcommand{\Lie} {\mathrm{Lie}}

\def\CC {{\mathbb C}}     %% complex numbers
\def \OO {{\mathcal O}}
\def \F {{\mathcal F}}

\def \g {\mathfrak{g}}

\newtheorem{theorem}{Theorem}[section]
\newtheorem{lemma}[theorem]{Lemma}

\newtheorem{prop}[theorem]{Proposition}

\newtheorem{coro}[theorem]{Corollary}

\begin{document}

\title{Decomposition of cohomology of vector bundles on homogeneous ind-spaces}

\author{Elitza Hristova\thanks{The first author has been partially supported by the Young Researchers Program of the Bulgarian Academy of Sciences, Project Number DFNP-86/04.05.2016 and by the Bulgarian National Science Fund under Grant I02/18.}, Ivan Penkov \thanks{The second author has been supported in part by DFG Grant PE 980/6-1.}}

\date{}

\maketitle

\begin{abstract} Let $G$ be a locally semisimple ind-group, $P$ be a parabolic subgroup, and $E$ be a finite-dimensional $P$-module. We show that, under a certain condition on $E$, the nonzero cohomologies of the homogeneous vector bundle $\OO_{G/P}(E^*)$ on $G/P$ induced by the dual $P$-module $E^*$ decompose as direct sums of cohomologies of bundles of the form $\OO_{G/P}(R)$ for (some) simple constituents $R$ of $E^*$. In the finite-dimensional case, this result is a consequence of the Bott-Borel-Weil theorem and Weyl's semisimplicity theorem. In the infinite-dimensional setting we consider, there is no relevant semisimplicity theorem. Instead, our results are based on the injectivity of the cohomologies of the bundles $\OO_{G/P}(R)$.
\end{abstract}

\section{Introduction}
The Bott-Borel-Weil theorem is a basic result which connects algebraic geometry with representation theory. More precisely, if $G$ is a connected complex reductive linear algebraic group and $P$ is a parabolic subgroup, the Bott-Borel-Weil theorem states that a homogeneous bundle on $G/P$, induced by a simple $P$-module, has at most one nonzero cohomology group. Moreover, this group is a simple highest weight $G$-module. Using Weyl's semisimplicity theorem, one can further show that, for any finite-dimensional $P$-module, the corresponding nonzero cohomologies are semisimple $G$-modules. More specifically, these cohomologies split as direct sum of cohomologies of homogeneous bundles induced by (some) simple constituents of the inducing $P$-module.

Analogues of the Bott-Borel-Weil theorem have been proved in various contexts. In particular, in \cite{DPW} and \cite{DP}, the case of locally reductive ind-groups $G$ has been studied. For a locally reductive ind-group $G$, a homogeneous bundle on $G/P$ induced by a simple finite-dimensional $P$-module still has at most one nonzero cohomology group. However, the difference with the finite-dimensional case is that this cohomology group is dual to a simple $G$-module, and hence is not irreducible being uncountable dimensional. In addition, the cohomology group may or may not be an integrable $\g = \Lie(G)$-module.

In this article we consider locally semisimple ind-groups $G$ and their homogeneous ind-spaces $G/P$ for parabolic ind-groups $P \hookrightarrow G$. To any finite-dimensional $P$-module $E$ we attach the homogeneous vector bundle $\OO_{G/P}(E^*)$ induced by the dual $P$-module $E^*$. We impose the condition that all non-zero cohomology groups $H^q(G/P,\OO_{G/P}(E^*))$ % of bundles of the form $\OO_{G/P}(R)$ for simple constituents $R$ of $E^*$ 
are integrable $\g$-modules.
%homogeneous vector bundles $\OO_{G/P}(E^*)$ of finite rank on $G/P$ such that the total cohomology $H^{\cdot}(G/P, \OO_{G/P}(E^*))$ is an integrable $\g = \Lie(G)$-module. We do not assume the $P$-module $E^*$ to be simple. 
Then, our main result (Theorem \ref{thm_DirectSum}) is that, despite the absence of a relevant analogue of Weyl's semisimplicity theorem, any such nonzero cohomology group $H^q(G/P,\OO_{G/P}(E^*))$ %of the bundle $\OO_{G/P}(E^*)$ 
is isomorphic to a direct sum of cohomologies of homogeneous bundles on $G/P$ induced by (some) simple constituents of $E^*$. The proof uses the fact, that any nonzero cohomology group of a homogeneous bundle on $G/P$ induced by a simple constituent of $E^*$ 
is injective in the category of integrable $\g$-modules.

\section{Preliminaries} \label{sec_Prelim}
In this section we summarize the definitions and properties from the field of ind-groups, which we need throughout the article. More detailed expositions on the subject can be found in \cite{DPW}, \cite{DP}.

A \textit{locally algebraic ind-group} over $\CC$, which we briefly refer to as an \textit{ind-group}, is a set $G$ which is the inductive limit of embeddings of connected algebraic groups, i.e. $G = \varinjlim G_n$, where
\begin{align} \label{eq_filtration}
G_1 \subset G_2 \subset \cdots G_n \subset \cdots
\end{align}
An ind-group $G$ is \textit{locally (semi)simple} %(resp., \textit{locally simple}) 
if the filtration (\ref{eq_filtration}) can be chosen so that each $G_n$ is a (semi)simple %(resp., simple) 
linear algebraic group. A first important class of locally simple ind-groups is that of diagonal locally simple ind-groups. For its definition we need to recall the notion of a diagonal embedding. An embedding of classical simple finite-dimensional groups $G' \subset G$ is called \textit{diagonal} if the induced injection of the Lie algebras $\g' \subset \g$ has the following property: the natural representation of $\g$ decomposes over $\g'$ as a direct sum of copies of the natural representation of $\g'$, of its dual, and of the trivial $\g'$-representation. A locally simple ind-group $G$ is called \textit{diagonal} if the filtration (\ref{eq_filtration}) can be chosen so that each embedding $G_n \subset G_{n+1}$ is a diagonal embedding of classical simple groups. %We say that $G$ is of type $A$, $B$, $C$, or $D$ if all $G_n$ can be chosen to be of the respective type $A$, $B$, $C$, or $D$. %It can be shown that if $G$ is diagonal, then each filtration of $G$ has a well-defined type $A$, $B$, $C$, or $D$. In general, a diagonal ind-group $G$ may have filtrations of different types, but we will use the term ``type of G'' to refer to the type of the fixed filtration.

First examples of diagonal ind-groups are the \textit{finitary simple ind-groups} $\SL(\infty) = \varinjlim \SL(n)$, $\SO(\infty) = \varinjlim \SO(n)$, and $\Sp(\infty) = \varinjlim \Sp(2n)$, where the inclusions $\SL(n) \subset \SL(n+1)$ and $\SO(n)\subset \SO(n+1)$ are given by $A \mapsto \left (\begin{smallmatrix} A & 0\\ 0 & 1\end{smallmatrix}  \right ) $ and the inclusion $\Sp(2n)\subset \Sp(2n+2)$ is given by $A \mapsto \left (\begin{smallmatrix} A & 0 & 0\\ 0 & 1 & 0 \\ 0 & 0 & 1\end{smallmatrix}\right )$.

Let $G = \varinjlim G_n$ be an ind-group. A \textit{Cartan (resp., Borel) subgroup} of $G$ is an ind-subgroup $H$ (resp., $B$) of $G$ such that for a well-chosen filtration (\ref{eq_filtration}), the group $H_n = G_n \cap H$ (resp., $B_n = G_n \cap B$) is a Cartan (resp., Borel) subgroup of $G_n$ for each $n$. A \textit{parabolic subgroup} $P$ of $G$ is an ind-subgroup $P$ which contains a Borel subgroup $B$. A \textit{$G$-module} is a vector space $V$ equipped with compatible structures of $G_n$-modules for all $n$.

An ind-group is an example of the more general notion of an ind-variety (see \cite{Ku}, \cite{DPW}). Briefly, an \textit{ind-variety} $X = \varinjlim X_n$ is determined by a sequence of closed embeddings of algebraic varieties
$$
X_1 \subset X_2 \subset \cdots \subset X_n \subset \cdots
$$
An ind-variety $X$ is automatically a topological space, and one defines the structure sheaf $\OO_X$ of $X$ as the projective limit sheaf of the structure sheaves $\OO_{X_n}$ of $X_n$. More generally, a sheaf $\F$ on $X$ is the limit of a projective system of sheaves $\F_n$ on $X_n$.

Let $G= \varinjlim G_n$ be a locally semisimple ind-group and $P = \varinjlim P_n$ be a parabolic subgroup. Let $E$ be a finite-dimensional $P$-module. %We fix exhaustions $G = \varinjlim G_n$ and $P = \varinjlim P_n$. %, and $E = \varinjlim E_n$. 
The inductive limit $G/P = \varinjlim G_n /P_n$ has a natural structure of an ind-variety. The sheaf $\OO_{G/P}(E^*)$ is defined as the projective limit $\OO_{G/P}(E^*) = \varprojlim \OO_{G_n/P_n}(E_n^*)$, where $E_n = E_{\vert P_n}$ and $\OO_{G_n/P_n}(E_n^*)$ is the sheaf of regular local sections of the homogeneous vector bundle over $G_n/P_n$ induced by the module $E_n^*$. Using the Mittag-Leffler principle, it is shown in \cite{DPW} that the cohomology group $H^q(G/P, \OO_{G/P}(E^*))$ is canonically isomorphic to the projective limit $\varprojlim H^{q}(G_n/P_n, \OO_{G_n/P_n}(E_n^*))$ for each $q$.

By definition (see \cite{DPW}), the Lie algebra of an ind-group $G = \varinjlim G_n$ is the inductive limit Lie algebra $\g = \varinjlim \g_n$ for the filtration
$$
\g_1 \subset \g_2 \subset \cdots \subset \g_n \subset \cdots,
$$
where $\g_n$ is the Lie algebra of $G_n$ and the inclusions $\g_n \subset \g_{n+1}$ are the differentials of the group embeddings $G_n \subset G_{n+1}$. A Lie algebra $\g$ which is isomorphic to a direct limit of finite-dimensional Lie algebras is called \textit{locally finite}. Therefore, the Lie algebra of an ind-group is locally finite. A locally finite Lie algebra $\g$ which is isomorphic to a direct limit of (semi)simple %(resp., semisimple) 
finite-dimensional Lie algebras, is called \textit{locally (semi)simple}%(resp., \textit{locally semisimple})
. First examples of locally simple Lie algebras are the \textit{simple finitary} Lie algebras $\sl(\infty)$, $\so(\infty)$, and $\sp(\infty)$, which are the Lie algebras of the ind-groups $\SL(\infty)$, $\SO(\infty)$, and $\Sp(\infty)$. 

%The structure and representation theory of locally semisimple Lie algebras is nowadays very well developed. In particular 
%In \cite{PS} the authors introduce several categories of modules over a locally semisimple Lie algebra $\g$. The first and largest of these categories is the category $\Int_{\g}$ of integrable $\g$-modules. 
A module $M$ over $\g$ is called \textit{integrable} if $\dim \mathrm{span} \{m, g\cdot m, g^2 \cdot m, \cdots \} < \infty$ for any $m \in M$ and $g \in \g$. Below we use extensively the properties of integrable $\g$-modules established in \cite{PS}. By $\Int_{\g}$ we denote the category of integrable $\g$-modules.

\section{On integrable $\g$-modules} \label{sec_Alg}

In this section $\g = \varinjlim \g_n$ denotes a locally semisimple Lie algebra.

\begin{lemma} \label{lemma_finLen} Let $M = \varinjlim M_n$ be an integrable $\g$-module obtained as the inductive limit of finite-dimensional $\g_n$-modules $M_n$. If the length of $M_n$ is bounded by a number $N$ not depending on $n$, the length of the $\g$-module $M$ is also bounded by $N$.
\end{lemma}

\begin{proof}
Assume, on the contrary, that $M$ possesses a chain of submodules of the form
$$
0 \subsetneq U^1 \subsetneq U^2 \subsetneq \dots \subsetneq U^{N} \subsetneq U^{N+1} = M.
$$ 
We choose elements $x_1 \in U^1 \setminus \{0\}$, $x_2 \in U^2 \setminus U^1$, $\dots$$, x_{N+1} \in U^{N+1} \setminus U^N$, and choose $n$ large enough so that $x_1, x_2, \dots, x_{N+1} \in M_n$. Then we consider the filtration of $M_n$
$$
0 \subset (U^1 \cap M_n) \subset (U^2\cap M_n) \subset \dots \subset (U^{N}\cap M_n) \subset M_n
$$
and note that $U^i \cap M_n \neq U^{i+1}\cap M_n$ for $i = 1, \dots, N$.
This contradicts the assumption that the length of $M_n$ is bounded by $N$. %Thus $M$ has also length at most $N$. 
\end{proof}

We recall the following results from \cite{PS} which are very useful for our considerations.
\begin{lemma} \label{lemma_PS} \cite[Lemma 4.1, Proposition 3.2]{PS} Let $M$ be an integrable $\g$-module.
\begin{itemize}
\item[(a)]The $\g$-module $M^* = \Hom_{\CC}(M, \CC)$ is an integrable $\g$-module if and only if for any $n$, $M$ considered as a $\g_n$-module has finitely many isotypic $\g_n$-components.
\item[(b)] If the $\g$-module $M^*$ is integrable, then $M^*$ is an injective object in the category $\Int_{\g}$.
\end{itemize}
\end{lemma}

\begin{coro} \label{prop_Decomp} Let $M \in \Int_{\g}$ be a module of finite length $N$ with simple constituents $M_i$, $i = 1, \dots N$, and such that $M^* \in \Int_{\g}$. Then there is an isomorphism of $\g$-modules 
$$M^* \cong \bigoplus_i M^*_i.$$
\end{coro}

\begin{proof} The case $N=1$ is clear and the general case follows from a simple induction argument using Lemma \ref{lemma_PS} (b) and the fact that $\Int_{\g}$ is closed with respect to taking submodules and quotients.
\end{proof}

In the setting of Section \ref{sec_Prelim}, let $B_n$ be a Borel subgroup of $G_n$. A weight $\mu$ of $B = \varinjlim B_n$ is the (projective) limit of a projective system of integral $G_n$-weights $\{\mu_n\}$. By definition, $\mu$ is $B$-\textit{dominant} if all $\mu_n$ are $B_n$-dominant. Let $\mu$ be a $B$-dominant weight. By $V_B(\CC_{\mu})$ we denote the inductive limit of finite-dimensional simple $G_n$-modules $V_{B_n}(\CC_{\mu_n})$ with $B_n$-highest weight $\mu_n$, where the highest weight space $\CC_{\mu_n}$ of $V_{B_n}(\CC_{\mu_n})$ is mapped to $\CC_{\mu_{n+1}}$. Assuming that $G_n$ is classical simple, following Bourbaki \cite{B}, we write the weights of $G_n$ as linear combinations of standard functions $\varepsilon_n^1, \dots, \varepsilon_n^{r_n +1}$ if $G_n$ is of type $A$, and $\varepsilon_n^1, \dots, \varepsilon_n^{r_n}$ otherwise, where $r_n = \mathrm{rk} \g_n$.

\begin{prop} \label{prop_FinitaryCaseIntegrable} Let $G$ be one of the finitary simple ind-groups and let $\mu = \varprojlim \mu_n$ be a dominant weight of $G$.
\begin{itemize}
\item[(i)] If $G \cong \SL(\infty)$, then $V_B(\CC_{\mu})^*$ is an integrable $\g$-module if and only if there exists an integer $m_0$ such that for any $n$ if $\mu_n = \sum_i a_n^i \varepsilon_n^i$ then $a_n^{1} - a_n^{r_n +1} \leq m_0$.
\item[(ii)] If $G \cong \SO(\infty)$ or $\Sp(\infty)$, then $V_B(\CC_{\mu})^*$ is an integrable $\g$-module if and only if there exists an integer $m_0$ such that for any $n$ if $\mu_n = \sum_i a_n^i \varepsilon_n^i$ then $a_n^{1} \leq m_0$.
\end{itemize}
\end{prop}

\begin{proof}
According to Lemma \ref{lemma_PS} (a), the $\g$-module $V_B(\CC_{\mu})^*$ is integrable if and only if, when considered as a $\g_n$-module, it has finitely many isotypic components. Using the standard branching rules for embeddings of classical groups (see e.g. \cite{Z}) one can show that this latter condition is equivalent to the explicit conditions of (i) and (ii).
\end{proof}

Notice that the simple tensor $\sl(\infty)$-, $\so(\infty)$-, $\sp(\infty)$-modules discussed in \cite{PS} are a special case of the modules discussed in Proposition \ref{prop_FinitaryCaseIntegrable}.

More generally, if $G$ is a diagonal nonfinitary locally simple ind-group, it turns out that the condition of $\g$-integrability of the module $V_B(\CC_{\mu})^*$ is equivalent to the condition that $\mathrm{Ann}_{\mathrm{U}(\g)}(V_{B}(\CC_{\mu})) \neq 0$, see \cite[Proposition 4.5]{PP} and the references therein. Here we present an explicit sufficient condition for the integrability of $V_B(\CC_{\mu})^*$ for all diagonal locally simple ind-groups.

\begin{prop} \label{prop_Int}
Let $G$ be a diagonal locally simple ind-group and let $\mu = \varprojlim \mu_n$ be a dominant weight of $G$. Assume that there exists an integer $m_0$ with the property that for any $n$, the weight $\mu_n$ has an expression $\mu_n = \sum_i a_n^i \varepsilon_n^i$ such that $\sum_i |a_n^i| \leq m_0$. Then $V_B(\CC_{\mu})^*$ is integrable as a $\g$-module.
\end{prop}

\begin{proof}  
%We proceed as in the second part of the previous proof. 
Since $G$ is diagonal, for large enough $n$ and for all $k$ the embeddings $\g_n \subset \g_{n+k}$ are diagonal. Therefore, we can decompose $V_{B_{n+k}}(\CC_{\mu_{n+k}})$ over $\g_n$ using branching rules for diagonal embeddings of classical Lie algebras, see \cite{HTW}, \cite{Hr}. 
These branching rules show that for the highest weight $\nu$ of any simple $\g_n$-module $U$ which enters the decomposition of $V_{B_{n+k}}(\CC_{\mu_{n+k}})$ over $\g_n$, there exists an expression $\nu = \sum_i b_i \varepsilon_n^i$ which satisfies $\sum_i|b_i| \leq \sum_i|a_{n+k}^i| \leq m_0$. This holds for any $k$. Thus there are only finitely many isotypic $\g_n$-components in the decomposition of $V_{B_{n+k}}(\CC_{\mu_{n+k}})$ and their number stabilizes for large $k$. By Lemma \ref{lemma_PS} (a), $V_B(\CC_{\mu})^*$ is an integrable $\g$-module.
\end{proof}

\section{Decomposition of cohomology}

In this section we fix $G = \varinjlim G_n$ to be a locally semisimple ind-group and $P  = \varinjlim P_n$ a parabolic subgroup. By $\g$ we denote the Lie algebra of $G$.

\begin{lemma} \label{lemma_DualModule} %Let $G = \varinjlim G_n$ be a locally semisimple ind-group, let $P = \varinjlim P_n$ be a parabolic subgroup and $E$ be a finite-dimensional (rational) $P$-module. 
Let $E$ be a finite-dimensional %(rational) 
$P$-module.
Then $H^{q}(G/P, \OO_{G/P}(E^*))$ is the dual module of an integrable $\g$-module of finite length.
\end{lemma}

\begin{proof}
By Theorem 10.3 in \cite{DPW} we have
$$
H^{q}(G/P, \OO_{G/P}(E^*))  = \varprojlim H^{q}(G_n/P_n, \OO_{G_n/P_n}(E_n^*)),
$$
where $E_n:= E_{\vert P_n}$. The classical Bott-Borel-Weil theorem implies that $H^{q}(G_n/P_n, \OO_{G/P}(E_n^*)) \cong M_n^*$, where $M_n$ is some finite-dimensional %(rational) 
$G_n$-module. In other words, 
$$
H^{q}(G/P, \OO_{G/P}(E^*)) = \varprojlim M_n^*,
$$
%and in this way $H^{q}(G/P, \OO_{G/P}(E^*))$ is the dual module to the module $M = \varinjlim M_n$. Obviously, $M \in \Int_{\g}$. Since the length of each $M_n$ is bounded by $N$, where $N$ denotes the length of $E$, Lemma \ref{lemma_finLen} implies that the length of $M$ is also bounded by $N$.
hence $H^{q}(G/P, \OO_{G/P}(E^*))$ is the dual module to the module $M = \varinjlim M_n$. Obviously, $M \in \Int_{\g}$. Since the length of each $M_n$ is bounded by the length of $E$, by Lemma \ref{lemma_finLen} the same holds for the length of $M$.
\end{proof}

\begin{lemma} \label{lemma_Inj} For any finite-dimensional $P$-module $E$, the integrability of the $\g$-module $H^q(G/P, \OO_{G/P}(E^*))$ implies the injectivity of this module in the category $\Int_{\g}$.
\end{lemma}
\begin{proof}
The statement follows directly from Lemma \ref{lemma_DualModule} and Lemma \ref{lemma_PS} (b).
\end{proof}

Let $E$ be a finite-dimensional irreducible $P$-module. 
Then Theorem 11.1 (ii) from \cite{DPW} states that $H^q(G/P, \OO_{G/P}(E^*)) \neq 0$ for at most one integer $q \geq 0$. Moreover, if $q$ is such an integer then $H^q(G/P, \OO_{G/P}(E^*)) = V^*$ for some irreducible %rational 
$G$-module $V$. We call the module $E$ \textit{strongly finite} if $V^*$ is an integrable $\g$-module.

Furthermore, we call an arbitrary finite-dimensional $P$-module $E$ \textit{strongly finite} if all its simple constituents are strongly finite.  

We now describe several classes of strongly finite modules. 
Let $G = \varinjlim G_n$ be a diagonal locally simple ind-group, let $H = \varinjlim H_n$ be a Cartan subgroup and $B = \varinjlim B_n$ be a Borel subgroup containing $H$. 
 
\begin{prop} \label{prop_StrFiniteDiagCase}%Let $E$ be a finite-dimensional $B$-module satisfying the following condition: if $\lambda = \varprojlim \lambda_n$ is a weight of a simple constituent of $E$, then there exists an integer $m_0$ such that $|\lambda_n|\leq m_0$ for all $n$. Then $E$ is a strongly finite $B$-module.
Let $E$ be a finite-dimensional $B$-module satisfying the following condition: if $\lambda = \varprojlim \lambda_n$ is a weight of a simple constituent of $E$, then there exists an integer $m_0$, such that for any $n$ the weight $\lambda_n$ has an expression $\lambda_n = \sum_i a_n^i \varepsilon_n^i$ with $\sum_i |a_n^i| \leq m_0$. Then $E$ is a strongly finite $B$-module.
\end{prop}

\begin{proof}
Let $L = \CC_{\lambda}$ denote the one-dimensional $B$-module of weight $\lambda = \varprojlim \lambda_n$, such that $\lambda$ satisfies the condition of the proposition. In \cite{DP} an analogue $W_B$ of the Weyl group for diagonal locally simple ind-groups is defined and it is proved that (see Theorem 4.27 in \cite{DP})
$$
H^q(G/B, \OO_{G/B}(L^*)) = V_B(\CC_{w \cdot \lambda})^*,
$$
where $w \in W_B$ and $w\cdot \lambda$ is a dominant weight defined as the (projective) limit of a projective system of $G_n$-weights $\{w(n)(\lambda_n + \rho_n) - \rho_n\}$. Here $w(n)$ is an element of the Weyl group of $G_n$, determined uniquely by $w$, and $\rho_n$ denotes as usual the half-sum of the positive roots of $\g_n$. The definitions of $w$ and of $w(n)$ can be found in \cite{DP}. It is important for our considerations that given $w$, there exists an integer $N$ such that %the following holds. If $G$ is of type $A$, then for each $n$, $w(n)$ is a permutation of at most $N$ of the $\varepsilon_n^i$'s. Respectively, if $G$ is of type $B$, $C$, or $D$, then for each $n$, $w(n)$ is a signed permutation of at most $N$ of the $\varepsilon_n^i$'s. Therefore, 
for each $n$, the weight $w(n)\rho_n - \rho_n$ is a linear combination of at most $N$ of the $\varepsilon_n^i$'s. Furthermore, it is proved in \cite{DP} that $\{w(n)\rho_n - \rho_n\}$ is a projective system of weights of $G_n$. It follows that, there exists an integer $m$, such that for all $n$ the weight $w(n)\rho_n - \rho_n$ has an expression $w(n)\rho_n - \rho_n = \sum_i b_n^i \varepsilon_n^i$ with $\sum_i |b_n^i| \leq m$. This shows that if
%In addition, since $w(n)$ is an element of the Weyl group of $G_n$, there exists an expression $w(n)\lambda_n = \sum_i k_n^i \varepsilon_n^i$ with $\sum_i |k_n^i| = \sum_i |a_n^i|$. Therefore, for each $n$ we have
$w(n)(\lambda_n + \rho_n) - \rho_n = \sum_i (k_n^i + b_n^i) \varepsilon_n^i
$
then 
$\sum_i |k_n^i + b_n^i| \leq m_0 + m.
$
Therefore, by Proposition \ref{prop_Int}, $V_B(\CC_{w \cdot \lambda})^*$ is an integrable $\g$-module. 
\end{proof}

When $G$ is a finitary simple ind-group the following stronger result holds.
\begin{prop}
Let $G$ be one of the finitary simple ind-groups and let $E$ be a finite-dimensional $B$-module. Then $E$ is strongly finite if and only if the following holds: 
\begin{itemize}
\item If $G \cong \SL(\infty)$ and $\lambda = \varprojlim \lambda_n$ is the weight of a simple constituent of $E$, then there exists an integer $m_0$ such that for any $n$, if $\lambda_n = \sum_i a_n^i \varepsilon_n^i$ then $a_n^{max} - a_n^{min} \leq m_0$. Here $a_n^{max} = \mathrm{max}_i\{a_n^i\}$ and $a_n^{min} = \mathrm{min}_i\{a_n^i\}$.
\item If $G \cong \SO(\infty)$ or $\Sp(\infty)$ and $\lambda = \varprojlim \lambda_n$ is the weight of a simple constituent of $E$, then there exists an integer $m_0$ such that for any $n$, if $\lambda_n = \sum_i a_n^i \varepsilon_n^i$ then $\mathrm{max}_i\{|a_n^i|\} \leq m_0$.
\end{itemize}
\end{prop}

The proof follows the same ideas as in Proposition \ref{prop_StrFiniteDiagCase}, and uses also Proposition \ref{prop_FinitaryCaseIntegrable}. 

We now return to the general setting of this section and state the main result of this note.

\begin{theorem} \label{thm_DirectSum} %Let $G = \varinjlim G_n$ be any locally semisimple ind-group and $P = \varinjlim P_n$ be a parabolic subgroup. 
Let $E$ be a strongly finite $P$-module. Then 
\begin{itemize}
\item [(i)] $H^q(G/P, \OO_{G/P}(E^*))$ is an integrable $\g$-module for any $q$, and $H^q(G/P, \OO_{G/P}(E^*)) \neq 0$ for at most $N$ values of $q$, where $N$ denotes the length of $E$.
\item [(ii)] $H^q(G/P, \OO_{G/P}(E^*)) \neq 0$ implies that as a $\g$-module $H^q(G/P, \OO_{G/P}(E^*))$ is isomorphic to a direct sum of cohomologies of homogeneous bundles of the form $\OO_{G/P}(R)$ for simple constituents $R$ of $E^*$.
\end{itemize}
\end{theorem}

\begin{proof}
 We use induction on $N$. When $E$ is irreducible, Theorem 11.1 from \cite{DPW} tells us that $\OO_{G/P}(E^*)$ has at most one non-vanishing cohomology group. The integrability of the nonzero cohomology group follows from the assumption that $E$ is strongly finite. This proves (i) and (ii) for an irreducible $P$-module $E$. Next, assume that (i) and (ii) hold for stongly finite modules $E'$ of length $N-1$. Then we consider our module $E$ of length $N$ together with a short exact sequence
$$
0 \rightarrow E'' \rightarrow E \rightarrow E' \rightarrow 0
$$ 
such that $E'$ is of length $N-1$ and $E''$ is irreducible. If $H^q(G/P, \OO_{G/P}(E''^*))  = 0$ for all $q$, both (i) and (ii) follow trivially from the induction assumption.
Suppose that  $H^q(G/P, \OO_{G/P}(E''^*)) \neq 0$. Then the short exact sequence 
$$
0 \rightarrow \OO_{G/P}(E'^*) \rightarrow \OO_{G/P}(E^*) \rightarrow \OO_{G/P}(E''^*) \rightarrow 0
$$
yields an exact sequence
\begin{equation} \label{eq_es}
\begin{aligned} 
&0 \rightarrow H^{q}(G/P, \OO_{G/P}(E'^*))\rightarrow H^{q}(G/P, \OO_{G/P}(E^*)) \rightarrow \\
&H^{q}(G/P, \OO_{G/P}(E''^*)) \stackrel{g}{\rightarrow} H^{q+1}(G/P, \OO_{G/P}(E'^*)) \rightarrow H^{q+1}(G/P, \OO_{G/P}(E^*)) \rightarrow 0.
\end{aligned}
\end{equation}
Since $\Int_{\g}$ is an abelian category and $E$ is strongly finite, it follows that $H^{q}(G/P, \OO_{G/P}(E^*))$ and $H^{q+1}(G/P, \OO_{G/P}(E^*))$ are integrable $\g$-modules. %Moreover, if $H^{q+1}(G/P, \OO_{G/P}(E'^*)) = 0$ then $H^{q+1}(G/P, \OO_{G/P}(E^*)) = 0$. 
Moreover, for $p \neq q, q+1$ we have $H^{p}(G/P, \OO_{G/P}(E^*)) \cong H^{p}(G/P, \OO_{G/P}(E'^*))$, and this proves (i). 

To prove (ii) we use Lemma \ref{lemma_connectHom} below and consider two cases for the connecting homomorphism $g$ in the exact sequence (\ref{eq_es}).

{\bf Case 1.} $\ker g = H^{q}(G/P, \OO_{G/P}(E''^*))$. Then the integrability of $H^{q}(G/P, \OO_{G/P}(E^*))$ proven in (i), and Lemmas \ref{lemma_DualModule}, \ref{lemma_Inj} imply that 
$$H^{q}(G/P, \OO_{G/P}(E^*))\cong H^{q}(G/P, \OO_{G/P}(E'^*)) \oplus H^{q}(G/P, \OO_{G/P}(E''^*)).
$$ Furthermore, $H^{q+1}(G/P, \OO_{G/P}(E'^*)) \cong H^{q+1}(G/P, \OO_{G/P}(E^*))$. Therefore, by induction (ii) holds for both $H^{q}(G/P, \OO_{G/P}(E^*))$ and $H^{q+1}(G/P, \OO_{G/P}(E^*))$. 

{\bf Case 2.} $\ker g = 0$. Then $\im g$ is a direct summand of $H^{q+1}(G/P, \OO_{G/P}(E'^*))$ by Lemma \ref{lemma_Inj}, i.e. 
$$
H^{q+1}(G/P, \OO_{G/P}(E'^*)) \cong H^{q}(G/P, \OO_{G/P}(E''^*)) \oplus H^{q+1}(G/P, \OO_{G/P}(E^*)).
$$
Furthermore, $H^{q}(G/P, \OO_{G/P}(E'^*))\cong H^{q}(G/P, \OO_{G/P}(E^*))$.

Since by induction both $H^{q}(G/P, \OO_{G/P}(E'^*))$ and $H^{q+1}(G/P, \OO_{G/P}(E'^*))$ are isomorphic to direct sums of cohomologies of bundles $\OO_{G/P}(R')$ for simple constituents $R'$ of $(E')^*$, and hence also of $E^*$, the same follows for $H^{q}(G/P, \OO_{G/P}(E^*))$ and $H^{q+1}(G/P, \OO_{G/P}(E^*))$.
\end{proof}

\begin{lemma} \label{lemma_connectHom} Consider the connecting homomorphism 
$$
g: H^q(G/P, \OO_{G/P}(E''^*)) \rightarrow H^{q+1}(G/P, \OO_{G/P}(E'^*)),
$$
where $E''$ is an irreducible $P$-module. Then $\ker g = H^q(G/P, \OO_{G/P}(E''^*))$ or $\ker g = 0$.
%Then exactly one of the following is true
%\begin{itemize}
%\item [(i)] $\ker g = H^q(G/P, \OO_{G/P}(E''^*))$;
%\item [(ii)] $\ker g = 0$.
%\end{itemize}
\end{lemma}

\begin{proof}
We have the following commuting diagram of homomorphisms
\begin{equation*}
\xymatrix{
&\dots \ar[r] & H^q(G_{n+1}/P_{n+1}, \OO_{G_{n+1}/P_{n+1}}(E_{n+1}''^*)) \ar[d]^{g_{n+1}}\ar[r]&H^q(G_{n}/P_{n}, \OO_{G_{n}/P_{n}}(E_{n}''^*)) \ar[d]^{g_n}\ar[r]&\dots&\\
&\dots \ar[r] & H^{q+1}(G_{n+1}/P_{n+1}, \OO_{G_{n+1}/P_{n+1}}(E_{n+1}'^*))\ar[r]&H^{q+1}(G_{n}/P_{n}, \OO_{G_{n}/P_{n}}(E_{n}'^*))\ar[r]&\dots&\\
}
\end{equation*}
where $H^q(G/P, \OO_{G/P}(E''^*)) = \varprojlim H^q(G_{n}/P_{n}, \OO_{G_{n}/P_{n}}(E_{n}''^*))$, $H^{q+1}(G/P, \OO_{G/P}(E'^*)) = \varprojlim H^{q+1}(G_{n}/P_{n}, \OO_{G_{n}/P_{n}}(E_{n}'^*))$, and $g = \varprojlim g_n$. 

For each $n$ it holds that $H^q(G_{n}/P_{n}, \OO_{G_{n}/P_{n}}(E_{n}''^*))$ is an irreducible $G_n$-module. Hence for $g_n$ we have $\ker g_n = H^q(G_n/P_n, \OO_{G_n/P_n}(E_n''^*))$ or $\ker g_n = 0$. By the commutativity of the above diagram the statement of the lemma follows.
\end{proof}

Theorem \ref{thm_DirectSum} together with Lemma \ref{lemma_Inj} imply
\begin{coro} If $E$ is strongly finite then all nonzero cohomology groups $H^q(G/P, \OO_{G/P}(E^*))$ are injective objects of $\Int_{\g}$. 
\end{coro}

\vspace{0.4cm}
\small{
\author{\noindent Elitza Hristova, Institute of Mathematics and Informatics, Bulgarian Academy of Sciences, Acad. Georgi Bonchev Str., Block 8, 1113 Sofia, Bulgaria. \\ E-mail: e.hristova@math.bas.bg}\\

\author{\noindent Ivan Penkov, Jacobs University Bremen, School of Engineering and Science, Campus Ring 1, 28759 Bremen, Germany. \\ E-mail: i.penkov@jacobs-university.de}
}

\end{document}